\documentclass[reqno]{amsart}

\usepackage{amsmath,amssymb,amsthm}

\usepackage{tikz}
\usepackage{caption}
\usepackage{graphicx}
\usepackage{graphics}

\newtheorem*{thm}{Theorem}
\newtheorem{lemma}{Lemma}

\newcommand{\vol}{\operatorname{vol}}

\begin{document}

\title[]{An elementary proof of a lower bound\\ for the inverse of the star discrepancy}

\author[]{Stefan Steinerberger}
\address{Department of Mathematics, University of Washington, Seattle, WA 98195, USA} \email{steinerb@uw.edu}

\thanks{S.S. is supported by the NSF (DMS-2123224) and the Alfred P. Sloan Foundation.}

\begin{abstract} A central problem in discrepancy theory is the challenge of evenly distributing points $\left\{x_1, \dots, x_n \right\}$ in $[0,1]^d$. Suppose a set is so regular that for some $\varepsilon> 0$ and all $y \in [0,1]^d$ the sub-region $[0,y] = [0,y_1] \times \dots \times [0,y_d]$ contains a number of points nearly proportional to its volume and
$$\forall~y \in [0,1]^d \qquad  \left| \frac{1}{n} \# \left\{1 \leq i \leq n: x_i \in [0,y] \right\} - \vol([0,y]) \right| \leq \varepsilon,$$
how large does $n$ have to be depending on $d$ and $\varepsilon$? We give an elementary proof of the currently best known result, due to Hinrichs, showing that $n \gtrsim d \cdot \varepsilon^{-1}$.
\end{abstract}
\maketitle

\section{Introduction}
 Given a set of $n$ points $\left\{x_1, \dots, x_n \right\} \subset [0,1]^d$, the star discrepancy is a measure of regularity that is defined by the largest difference between the empirical distribution and the uniform distribution taken over all axis-parallel rectangles
$$ D_n^* = \max_{y \in [0,1]^d} \left| \frac{1}{n}\# \left\{1 \leq i \leq n: x_i \in [0,y] \right\} - \vol([0,y]) \right|,$$
where $[0,y] = [0,y_1] \times \dots \times [0,y_d]$. The main goal is to find sets of points for which this quantity is small. There are perhaps two main problems in this area: is it possible to find a sequence $(x_k)_{k=1}^{\infty}$ in $[0,1]^d$ such that the star discrepancy of the first $n$ elements is uniformly small in $n$? The best known constructions achieve the rate $(\log{n})^d n^{-1}$ as $n \rightarrow \infty$, however, the implicit constants are so large that $n$ has to be exponentially large in $d$ for these bounds to become effective.\\
This leads to the second problem, understanding the possible size of the star discrepancy in the regime where $n$ is polynomial in $d$. Suppose we want to ensure that $D_n^* \leq \varepsilon$, how large does $n$ have to be (depending on $d$ and $\varepsilon$)? The smallest possible cardinality of a set with this property is denoted by
$N_{\infty}^*(d, \varepsilon)$. Heinrich, Novak, Wasilkowski \& Wozniakowski \cite{heinrich} showed that, for some universal $c>0$, 
$$ N_{\infty}^*(d, \varepsilon) \leq c \frac{d}{\varepsilon^{2}}.$$
Aistleitner \cite{aistleitner} later proved that $c=10$ is admissible. The constant was further improved \cite{doerr2, gne2, pasing}, the current record is $c \leq 2.4968$ due to Gnewuch, Pasing \& Wei\ss ~\cite{gne3}. These constructions are all probabilistic. Doerr \cite{doerr} showed that this rate is optimal up to constants when working with iid uniformly distributed points.  This is a common theme in combinatorics: random objects behave in a certain way and the task at hand is to understand whether structured constructions can lead to a better performance. This leads to the study of lower bounds for $N_{\infty}^*(d, \varepsilon)$. Heinrich, Novak, Wasilkowski \& Wozniakowski \cite{heinrich} originally showed
$$  N_{\infty}^*(d, \varepsilon) \geq c \cdot  d \cdot \log\left(\frac{1}{\varepsilon}\right).$$
This was then improved by Hinrichs to what is still the best known result.
\begin{thm}[Hinrichs \cite{aicke}] There are universal $c, \varepsilon_0 > 0$ such that 
$$ \forall~0 < \varepsilon < \varepsilon_0  \qquad N_{\infty}^*(d, \varepsilon) \geq c \frac{d}{\varepsilon^{}}.$$
\end{thm}
 The original proof is highly nontrivial and uses Vapnik-Chervonenkis classes, the Sauer-Shelah Lemma \cite{sauer, shel} and a packing argument. However, it is also slightly more general since it can be applied to weighted points and to more general geometries (which our argument, in its present shape, cannot). Another proof of the Theorem, as stated, was given by Aistleitner \& Hinrichs \cite{aicke3} based on elegant combinatorial considerations. The goal of this note is to give a very short and simple argument (showing $c \geq 1/(9e)$ in sufficiently high dimensions) in the hope of drawing attention to this very interesting problem.

\section{Proof}
\subsection{Overview.}
Let $X = \left\{x_1, \dots, x_n \right\} \subset [0,1]^d$ be given and assume that
\begin{align} \label{eq:cond}
\forall~y \in [0,1]^d \qquad  \left| \frac{1}{n} \# \left\{1 \leq i \leq n: x_i \in [0,y] \right\} -  \vol([0,y]) \right| \leq \varepsilon.
\end{align}
The goal is to construct a lower bound on $n$ based only on $\varepsilon$ and $d$.
The proof proceeds by constructing a sequence $y_0, y_1, \dots, y_k \in [0,1]^d$ such that
$$ \left[ 0, 1\right]^d = \left[ 0, y_0\right] \supset [0,y_1] \supset  [0,y_2] \dots \supset [0,y_k] \supseteq \left[ 0, 1-\frac{1}{d} \right]^d$$
with the property that
$ \# \left( X \cap  \left([0,y_i] \setminus [0,y_{i+1}] \right)\right) \geq 1.$
This clearly implies $n \geq k$.
 We will moreover specify that each $y_i$ can be written as
$$ y_i = \left( 1 - a_{i,1} \frac{20\varepsilon}{d}, 1- a_{i,2} \frac{20\varepsilon}{d}, \dots, 1- a_{i,d} \frac{20\varepsilon}{d} \right)$$
where $a_{i,j} \in \mathbb{N}$ and $a_{i,j} \leq 1/(20\varepsilon)$. Moreover, $y_{i+1}$ will arise from taking $y_i$ and increasing a specific index $a_{i,j}$ by 1.
 The remaining question is whether it is possible to obtain a large value of $k$ in this framework.
\begin{center}
\begin{figure}[h!]
\begin{tikzpicture}
\draw [thick] (0,0) -- (0,2) -- (2,2) -- (2,0) -- (0,0);
\draw [ultra thick] (0,0) -- (0,2) -- (2,2) -- (2,0) -- (0,0);
\draw [thick] (3,0) -- (5,0) -- (5,2) -- (3,2) -- (3,0);
\draw [ultra thick] (3,0) -- (4.5,0) -- (4.5,2) -- (3,2) -- (3,0);
\filldraw (4.75, 0.5) circle (0.04cm);
\draw [thick] (6,0) -- (8,0) -- (8,2) -- (6,2) -- (6,0);
\draw [ultra thick] (6,0) -- (7.5,0) -- (7.5,1.5) -- (6,1.5) -- (6,0);
\filldraw (7.75, 0.5) circle (0.04cm);
\filldraw (7.3, 1.75) circle (0.04cm);
\draw [thick] (9,0) -- (11,0) -- (11,2) -- (9,2) -- (9,0);
\draw [ultra thick] (9,0) -- (10,0) -- (10,1.5) -- (9,1.5) -- (9,0);
\filldraw (10.75, 0.5) circle (0.04cm);
\filldraw (10.3, 1.75) circle (0.04cm);
\filldraw (10.23, 1.2) circle (0.04cm);
\end{tikzpicture}
\caption{Sketch of the argument: step-wise shrinking of a box. Each step uncovers a new point.}
\end{figure}
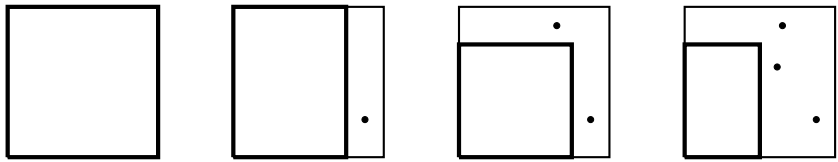
\end{center}

\subsection{Two Lemmata} We start with a simple volume comparison inequality: if we have two large boxes, one containing the other, with very different endpoints, then one contains a larger number of points than the other.
\begin{lemma}  Assume that $\left\{x_1, \dots, x_n\right\} \subset [0,1]^d$ satisfies \eqref{eq:cond}.
If $z,w \in [0,1]^d$ satisfy
$$ [0,z] \supset [0,w] \supseteq \left[0, 1- \frac{1}{d} \right]^d \qquad \mbox{and} \quad \| z - w\|_{\ell^1} \geq 9 \varepsilon,$$
then there is an element of $X$ contained in $[0,z] \setminus [0,w]$ and
$$ \# (X \cap \left( [0,z] \setminus [0,w] \right)) \geq 1.$$
\end{lemma}
\begin{proof} The proof is comprised of two parts. The first is to show that $ [0,z] \setminus [0,w]$ has large volume. We use an argument of Hinrichs \cite{aicke} to bound
\begin{align*}
 \vol([0,z] \setminus [0,w]) &= \prod_{i=1}^{d} z_i - \prod_{i=1}^{d} w_i = \sum_{i=1}^{d} z_1 \dots z_{i-1} (z_i - w_i) w_{i+1} \dots w_d \\
 &\geq \left( \prod_{i=1}^{d} w_i\right) \cdot \sum_{i=1}^{d} (z_i - w_i) \geq \left(1- \frac{1}{d}\right)^d \cdot \|z - w\|_{\ell^1} \geq \frac{1}{4} \|z-w\|_{\ell^1}.
 \end{align*}
For the second part, note that $[0,z]$ contains at least as many elements from $X$ as $[0,w]$, we want to show that it contains
 more. 
If $\|z-w\|_{\ell^1} \geq 9 \varepsilon$, then we have $ \vol([0,z] \setminus [0,w]) > 2 \varepsilon$
which implies that $[0,z] \setminus [0,w]$ intersects $X$.
\end{proof}

\begin{lemma} Assume that $\left\{x_1, \dots, x_n\right\} \subset [0,1]^d$ satisfies \eqref{eq:cond}. Let 
$$ y_i = \left( 1 - a_{i,1} \frac{20\varepsilon}{d}, 1- a_{i,2} \frac{20\varepsilon}{d}, \dots, 1- a_{i,d} \frac{20\varepsilon}{d} \right) \in [0,1]^d$$
and assume that $a_{i,j} \leq 1/(20\varepsilon)$ for all $1 \leq j \leq d$. Suppose moreover that there are at least $d/2$ indices satisfying $a_{i,j} \leq 1/(20\varepsilon) - 1$. Then there exists $1 \leq k \leq d$ such that  $a_{i,k} \leq 1/(20\varepsilon) - 1$ with the property that increasing $a_{i,k}$ by 1 $$ a_{i+1, j} = \begin{cases} a_{i,j} \qquad &\mbox{if}~j \neq k \\ a_{i,j}+ 1 \qquad &\mbox{if}~j=k \end{cases}$$
defines a point $y_{i+1} \in [0,1]^d$ with $[0,y_i] \supset [0,y_{i+1}] \supseteq [0,1-1/d]^d$ so that
$$ \# \left( X \cap  \left([0,y_i] \setminus [0,y_{i+1}] \right)\right) \geq 1.$$
\end{lemma}
\begin{proof} We argue by contradiction: let $j$ be an index such that $a_{i,j} \leq 1/(20\varepsilon) - 1$. We call an index $j$ \textit{bad} if increasing $a_{i,j}$ by 1 does not capture an additional point:
 $$ \# \left( X \cap  \left([0,y_i] \setminus   \left( 1 - a_{i,1} \frac{20\varepsilon}{d},\dots, 1- (a_{i,j}+1) \frac{20\varepsilon}{d}, \dots, 1- a_{i,d} \frac{20\varepsilon}{d} \right) \right)\right) =0.$$
 This implies that if we simultaneously increase \textit{all} $a_{i,j}$ corresponding to bad indices, we obtain a new box with exactly the same points. More formally, setting
$$ a_{i+1, j} = \begin{cases} a_{i,j} + 1 \qquad &\mbox{if}~j ~ \mbox{is bad} \\ a_{i,j} \qquad &\mbox{otherwise.} \end{cases}$$
leads to a new box with
$ \# \left( X \cap  \left([0,y_i] \setminus [0,y_{i+1}] \right)\right) = 0.$
Since there are at least $d/2$ indices smaller than $\leq 1/(20\varepsilon) - 1$, we deduce that if all of them were bad, then the new point $y_{i+1}$ satisfies
$$ \| y_{i+1} - y_i\|_{\ell^1} \geq \frac{d}{2} \frac{20 \varepsilon}{d} \geq 10 \varepsilon$$
which contradicts Lemma 1. Thus there exists at least one index $1 \leq k \leq d$ such that $a_{i,k} \leq 1/(20\varepsilon) - 1$ and so that $k$ is not bad. This proves the result.
\end{proof}

\begin{proof}[Proof of the Theorem] We apply Lemma 2 iteratively starting with $[0,y_0] = [0,1]^d$. Each application of Lemma 2 finds one new point in the set. Lemma 2 stops being applicable as soon as half of the indices are of size $1/(20\varepsilon)$. Thus there are at least
$$ n \geq k \geq  \frac{d}{2} \left\lfloor \frac{1}{20\varepsilon} \right\rfloor \gtrsim \frac{d}{\varepsilon} \qquad \mbox{points.}$$
\end{proof}

\textbf{Acknowledgment.} The author is grateful for helpful conversations with Christoph Aistleitner, Michael Gnewuch and Noah Kravitz.

\end{document}